\documentclass[english,a4paper,12pt]{amsart}

\usepackage[T1]{fontenc} 
\usepackage[latin1]{inputenc}
\usepackage{amssymb, babel, xcolor, tikz-cd}
\usepackage{amssymb}
\usepackage{amsmath}
\usepackage{bm}


\usepackage[T1]{fontenc} 
\usepackage[latin1]{inputenc}
\usepackage{amssymb, babel}
\usepackage{amssymb}
\usepackage{amsmath}

\usepackage{fullpage}
\usepackage{xcolor}

\date{}
\title{Embeddings of $E(1,6)$ in $E(5,10)$ and $E(4,4)$} 
\author{Nicoletta Cantarini}\author{Fabrizio Caselli}\author{Victor Kac}
\address{Fabrizio Caselli and Nicoletta Cantarini, Dipartimento di matematica, Universit\`a di Bologna, Piazza di Porta San Donato 5, 40126 Bologna, Italy}
\email{fabrizio.caselli@unibo.it}
\email{nicoletta.cantarini@unibo.it}
\address {Victor Kac, Department of Mathematics, MIT, 77 Massachusetts Avenue, Cambridge, MA 02139, USA}

\email{kac@math.mit.edu}

\newtheorem{theorem}{Theorem}[section] 
\newtheorem{lemma}[theorem]{Lemma} 
\newtheorem{corollary}[theorem]{Corollary} 
\newtheorem{proposition}[theorem]{Proposition} 
\newtheorem{definition}[theorem]{Definition} 
\newtheorem{remark}[theorem]{Remark}

\def\de{\partial}
\def\Z{\mathbb{Z}} 

\def\g{\mathfrak{g}}

\def\C{\mathbb{C}}

\def\0{\bar{0}}
\def\1{\bar{1}}

\newcommand{\inlinewedge}{\textrm{\raisebox{0.6mm}{\footnotesize $\bigwedge$}}}
\newcommand{\displaywedge}{\textrm{\raisebox{0.6mm}{\tiny $\bigwedge$}}}

\DeclareMathOperator{\diver}{div}
\DeclareMathOperator{\ad}{ad}
\begin{document} 
\maketitle
\begin{abstract} We study the embeddings of the exceptional infinite-dimensional Lie superalgebra $E(1,6)$ in the exceptional Lie superalgebras $E(5,10)$ and $E(4,4)$. These questions arose in recent works on enhanced symmetries in some supersymmetric theories by N.\ Garner, S.\ Raghavendran, I. Saberi and B.\ Williams.
\end{abstract}
\section{Introduction}
The classification of simple linearly compact Lie superalgebras \cite{K} includes five exceptional ones, denoted by $E(1,6)$, $E(3,6)$, $E(3,8)$, $E(4,4)$, and $E(5,10)$ (notation $E(m,n)$ means that this Lie superalgebra can be realized most economically by vector fields in a formal neighborhood of a point in a $(m|n)$-dimensional supermanifold). A few papers appeared recently, where these Lie superalgebras occurred as symmetries of some theories, namely, $E(5,10)$ appeared in the 11-dimensional supergravity, $E(3,6)$ in the holographic approach to 6-dimensional superconformal index, and $E(1,6)$ in the enhanced $N=8$ supersymmetric Chern Simons theory, see references \cite{RSW}, \cite{RW},
\cite{G}, respectively.

From this perspective and in view of representation theory of linearly compact Lie superalgebras, it is therefore important to study how these Lie superalgebras "interact" with each other.

An embedding of $E(3,6)$ in $E(5,10)$ and the decomposition of $E(5,10)$ with respect to the adjoint representation of $E(3,6)$ on it was studied in \cite{KR}, and this decomposition was used in \cite{RW}.

In the present paper we, first, present a new description of the Lie superalgebra $E(1,6)$ (Section \ref{preliminary}). Then we study the embedding of $E(1,6)$ in $E(5,10)$ and the decomposition of $E(5,10)$ with respect to the adjoint representation of $E(1,6)$ on it (Theorem \ref{3.5}).

Next, we construct explicitly an embedding of $E(1,6)$ in $E(4,4)$ (Theorem
\ref{4.7}). This is achieved by exploiting new properties of the classical realization of $E(1,6)$ inside $K(1,6)$ (Theorem \ref{4.6}).  In this case, however, it seems a difficult problem to find the corresponding decomposition.

Note that neither $E(3,8)$, nor $E(4,4)$ can be embedded in any of the other five exceptional linearly compact Lie superalgebras \cite{CK}.

\section{The exceptional Lie superalgebras $E(5,10)$, $E(4,4)$ and $E(1,6)$ and their principal gradings.}\label{preliminary}
Let us recall the definition of the simple linearly compact Lie superalgebras 
$E(5,10)$, $E(4,4)$ and $E(1,6)$. For the description of these algebras and their $\Z$-gradings we refer to \cite{CK}, \cite{CaK} and \cite{K}.

We use the following notation. The linearly compact Lie algebra $W_n$ 
(resp.\ $S_n$) consists of all vector fields $\sum_{j=1}^nP_j(x)\partial_j$, where 
$P_j(x)\in\C[[x_1, \dots, x_n]]$, $\partial_j=\frac{\partial}{\partial x_j}$ (resp.\
of all zero-divergence vector fields in $W_n$). The Lie algebra $W_n$ acts on the linearly compact vector space $\Omega^k(n)^{\lambda}$ of all differential $k$-forms with coefficients in $\C[[x_1, \dots, x_n]]$
by the formula
\[X(\omega)=L_X\omega+\lambda(\diver X)\omega, \,\, X\in W_n,\,\,\omega\in 
\Omega^k(n),\,\,\lambda\in\C,\]
where $L_X\omega$ denotes the Lie derivative of the differential $k$-form $\omega$ along the vector field $X$. Finally, $\Omega^k(n)={\Omega^k(n)}^0$ (resp.\ 
$\Omega^k(n)_{cl}$) denotes the space of all (resp.\ all closed) differential $k$-forms in $n$ variables. We will adopt the simple notation $d_{i_1i_2\cdots i_k}$ to denote the differential $k$-form $dx_{i_1}\wedge dx_{i_2}\wedge \cdots dx_{i_k}$.

The Lie superalgebra $\g=E(5,10)$ has  even part $\g_{\bar{0}}=S_5$, and odd part $\g_{\bar{1}}=\Omega^2(5)_{cl}$.
The bracket between a vector field and a two-form is given by the Lie derivative, and on $\g_{\bar{1}}$  the bracket is
$$[fd_{ij},g d_{hk}]=\varepsilon_{ijkl}fg\partial_{[ijkl]},\,\, f,g\in \C[[x_1,\dots,x_5]],$$ 
where, for $i,j,k,l\in \{1,2,3,4,5\}$, $\varepsilon_{ijkl}$ and $[ijkl]$ are defined as follows: if $|\{i,j,k,l\}|=4$ we let $[ijkl]\in \{1,2,3,4,5\}$ be such that $|\{i,j,k,l,[ijkl]\}|=5$ and let $\varepsilon_{ijkl}$ be the sign of the permutation 
$(i,j,k,l,[ijkl])$; if $|\{i,j,k,l\}|<4$ then $\varepsilon_{ijkl}=0$.

The $\Z$-gradings of $\g=E(5,10)$, up to automorphisms, are  parametrized by quintuples of  integers $(a_1, a_2, a_3, a_4, a_5)$ with even sum, by letting 
\[\deg(x_i)=-\deg(\partial_i)=a_i,\,\, \mbox{and}\,\, \deg d=-\frac{1}{4}\sum_{i=1}^5a_i.\]
Such a grading is called  a grading of type  $(a_1, a_2, a_3, a_4, a_5)$. The grading of type
$(2,2,2,2,2)$ is called the principal grading.



The Lie superalgebra $L=E(4,4)$ is defined as follows: its even part $L_{\bar{0}}$ is the Lie algebra $W_4$  and its odd part $L_{\bar{1}}$ is isomorphic,  as an $L_{\bar{0}}$-module, to $\Omega^1(4)^{-\frac{1}{2}}$. The bracket of two odd elements $\omega_1, \omega_2\in L_{\bar{1}}$ is
$$[\omega_1, \omega_2]=d\omega_1\wedge \omega_2+\omega_1\wedge d\omega_2,$$
where the  differential 3-forms are identified with vector fields via contraction with the standard volume form.

The Lie superalgebra $L=E(4,4)$ has, up to automorphisms, only one irreducible $\Z$-grading $L=\prod_{j\geq -1}L_j$, called the principal grading, defined by setting
\[\deg x_i=1,\,\, \deg d=-2,\]
where $L_0$ is isomorphic to the unique non-trivial
central extension $\hat{p}(4)$ of the finite-dimensional Lie superalgebra $p(4)$
(see, for example, \cite{CK}, \cite{S} or \cite{CCK}). 

\medskip

Before recalling the construction of the Lie superalgebra $E(1,6)$ we
make the following observation. Let us denote by $(\Omega^k(n))_m$ the space of differential $k$-forms in $n$ variables $x_1,\dots, x_n$ with coefficients in the space $S^m\C^n$ of homogeneous polynomials of degree $m$ in the variables $x_1, \dots, x_n$. Then we can identify $(\Omega^k(n))_m$ with $S^m(\C^n)\otimes \displaywedge^k(\C^n)$ which decomposes, as an $\mathfrak{sl}_n$-module, for $0<k<n$, $m>0$, as follows (see, e.g. \cite[Table 5, line 4]{OV}):
\[S^m(\C^n)\otimes \displaywedge^k(\C^n)=V(m\lambda_1)\otimes V(\lambda_k)=V(m\lambda_1+\lambda_k)\oplus V((m-1)\lambda_1+\lambda_{k+1}),\]
where $\lambda_1, \dots, \lambda_n$ are the fundamental weights of $\mathfrak{sl}_n$ and $V(\lambda)$ denotes the highest weight module for 
$\mathfrak{sl}_n$ of weight $\lambda$.

For $\omega\in (\Omega^k(n))_m$
we set
\[{\textstyle \int}\omega=\frac{1}{k+m}i_E\omega,\]
where $E=\sum_{j}x_j\de_j$ is the Euler operator and $i_E\omega$ denotes the contraction of the form $\omega$ along the vector field $E$. Note that $\int$ defines a linear map $(\Omega^k(n))_m\rightarrow (\Omega^{k-1}(n))_{m+1}$ for $k\geq 1$.

\begin{proposition} \label{dint}Let $0<k<n$ and $m>0$. Then
\begin{itemize}
\item[a)] The $\mathfrak{sl}_n$-submodule $V(m\lambda_1+\lambda_k)$ of $(\Omega^k(n))_m$ is the image of the differential map $d:(\Omega^{k-1}(n))_{m+1}\rightarrow (\Omega^k(n))_m$.
\item[b)] The $\mathfrak{sl}_n$-submodule
$V((m-1)\lambda_1+\lambda_{k+1})$ of $(\Omega^k(n))_m$ is the image of the map ${\textstyle \int}:(\Omega^{k+1}(n))_{m-1}\rightarrow (\Omega^k(n))_m$.
\item[c)] The maps $d$ and $\int$ obey the following relations:
\[
d^2=0,\,\,\textstyle{\int}^2=0,\,\,d\textstyle{\int}+\textstyle{\int} d=1.
\]
\item [d)] If $d\omega=0$ (respectively $\int \!\omega=0$) then $d\int \!\omega=\omega$ (respectively $\int\! d\omega=\omega)$. 

\end{itemize}
\end{proposition}
\begin{proof}
	We first prove c). Since $i_Xi_Y=-i_Yi_X$ for all vector fields $X,Y$ we have $i_E^2=0$ and hence $\int^2=0$. Moreover, by Cartan's formula, if $\omega\in (\Omega^k(n))_m$, we have $(d\int + \int d)(\omega)=\frac{1}{k+m}L_E(\omega)=1$.
	
	Note that the closed form $x_1^m dx_1\wedge\cdots \wedge dx_k$ is a highest weight vector of weight $m\lambda_1+\lambda_k$ and hence the component $V(m\lambda_1+\lambda_k)$ is given by the space of closed forms and a) follows.
	
	By c) we have $\int=\int d \int$ and $d \int d=d$ so $d$ and $\int$ are inverses to each other between the image of $d$ in $\Omega^{k+1}(n)_{m-1}$ and the image of $\int$ in $\Omega^k(n)_m$. In particular, by a), we have that the image of $\int$ is the other component $V((m-1)\lambda_1+\lambda_{k+1})$, with highest weight vector $\int(x_1^{m-1}dx_1\wedge \cdots \wedge dx_{k+1})$, whence b). The statement in d) follows immediately from c).
\end{proof}
The following diagram is an illustration of Proposition \ref{dint}:\vspace{5mm}

\begin{tikzcd}[column sep=tiny]
	(\Omega^{k+1}(n))_{m-1}=&V((m-1)\lambda_1+\lambda_{k+1}) \arrow[rrd, bend left=5,"\int" near end ]  & \bigoplus & V((m-2)\lambda_1+\lambda_{k+2})  \\ \hspace{8mm}(\Omega^k(n))_m=&V(m\lambda_1+\lambda_k) \arrow[rrd, bend left=5,"\int" near end ]  & \bigoplus & V((m-1)\lambda_1+\lambda_{k+1})  \arrow[llu,bend left=10, near end,"d"]\\ 
(\Omega^{k-1}(n))_{m+1}=&	V((m+1)\lambda_1+\lambda_{k-1}) & \bigoplus& V(m\lambda_1+\lambda_k).  \arrow[llu,bend left=10, near end,"d"] 
\end{tikzcd}

\vspace{5mm}
Now we recall the construction of the Lie superalgebra $E(1,6)$. This is slightly different but equivalent to that in \cite{CK}. The even part of $E(1,6)$ is 
\[E(1,6)_{\bar{0}}=W_1\oplus (\mathfrak{sl}_4\otimes\C[[t]])\,\,{\mbox{(direct sum of Lie algebras),}}\]
where $W_1=\C[[t]]\partial_t$,
and its odd part is 
\[ E(1,6)_{\bar{1}}=(S^2(\C^4)\otimes{\Omega^1(1)}^{-\frac{1}{2}})\oplus (\displaywedge^2(\C^4)\otimes{\Omega^1(1)}^{-\frac{3}{2}})\,{\mbox{(direct sum of vector spaces)}}.
\]
The action of $E(1,6)_{\bar{0}}$ on $ S^2(\C^4)\otimes{\Omega^1(1)}^{-\frac{1}{2}}$ and the action of $W_1$ on $ \inlinewedge^2(\C^4)\otimes{\Omega^1(1)}^{-\frac{3}{2}}$ are the obvious ones (thinking $\Omega^1(1)^\lambda$ as the $W_1$-module of 1-forms in the variable $t$ described in Section \ref{preliminary}), and the remaining brackets are described below.

In what follows, for reasons that will be apparent in the following sections, we identify $\C^4$ with $\sum_{i=2}^5\C x_i$, so that $S^2\C^4$ is identified with the space of quadratic forms in $x_2,\dots, x_5$, $\inlinewedge^2(\C^4)$ is identified with the space of differential 2-forms $\sum_{i,j=2}^5a_{ij}dx_i\wedge dx_j$,
where $a_{ij}=-a_{ji}\in \C$, and $\mathfrak{sl}_4(\C)$ is identified with the space of zero divergence vector fields $\sum_{i,j=2}^5b_{ij}x_i\partial_j$.
Then the bracket between an element $X\otimes f\in \mathfrak{sl}_4\otimes \C[[t]]$ and an element $\omega\otimes gdt\in \inlinewedge^2(\C^4)\otimes{\Omega^1(1)}^{-\frac{3}{2}}$ is
\[[X\otimes f,\omega\otimes gdt]={\textstyle \int} (i_X\omega)\otimes f'gdt+L_X\omega\otimes fgdt\in (S^2(\C^4)\otimes{\Omega^1(1)}^{-\frac{1}{2}})\oplus ( \displaywedge^2(\C^4)\otimes{\Omega^1(1)}^{-\frac{3}{2}}).\]
 The bracket between two odd elements is,
for $p, q\in S^2(\C^4)$, $\omega, \sigma\in\inlinewedge^2(\C^4)$ and $f,g\in \C[[t]]$,
\begin{align*}[p\otimes fdt, q\otimes gdt]&=0;\\
	[p\otimes fdt, \omega\otimes gdt]&=-(dp\wedge\omega)\otimes fg\in \mathfrak{sl}_4\otimes \C[[t]];\\
	[\sigma\otimes fdt, \omega\otimes gdt]&=fg(\sigma\wedge\omega) +\frac{1}{2}\big({\textstyle \int} \sigma\wedge \omega - \sigma \wedge {\textstyle \int} \omega\big)\otimes (f'g-fg')\in W_1\oplus (\mathfrak{sl}_4\otimes \C[[t]]),
\end{align*}
where we identified closed differential 3-forms in 4 variables $x_2, \dots, x_5$
(resp.\  differential 4-forms in 5 variables $t, x_2, \dots, x_5$) with zero-divergence vector fields in the same variables (resp.\ with vector fields in one variable) through contraction with the  standard volume form $dx_2\wedge dx_3\wedge dx_4\wedge dx_5$ (resp.\ $dt\wedge dx_2\wedge dx_3\wedge dx_4\wedge dx_5$).

If we set
$z=2t\partial_t$, then $\ad(z)$ defines on $L$
an irreducible, consistent $\Z$-grading 
of depth 2 that 
is called the principal grading of $E(1,6)$ (i.e., $x\in E(1,6)$ has degree $d$ if $\ad(z)(x)=dx$). In this grading $E(1,6)_0\cong \mathfrak{sl}_4\oplus \C$ and $E(1,6)_{-1}\cong \inlinewedge^2(\C^4)$, $E(1,6)_{-2}\cong\C$,
$E(1,6)_1\cong S^2(\C^4)\oplus (\inlinewedge^2(\C^4))^*$, as $E(1,6)_0$-modules.

\section{Embedding the Lie superalgebra $E(1,6)$ into $E(5,10)$}

Consider $\g=E(5,10)$ with the $\Z$-grading of type $(0,1,1,1,1)$, i.e.,
\[\deg x_1=\deg \de_{1}=0,\, \deg x_i=-\deg \de_{i}=1 \, \mbox{for}\,
i=2,\dots,5, \, \deg d=-1.\]
Then $\g=\prod_{i\geq -1} \g_i$ is a $\Z$-graded Lie superalgebra of depth 1 by
closed subspaces $\g_i$, and the subalgebra
$\g_0$ is isomorphic to $E(1,6)$ (see \cite{CK}). In Theorem \ref{g00}
below we give a precise description of this isomorphism. We have
\begin{align*}(\g_0)_{\bar{0}}=&\{X\in\langle \partial_{1}, x_i\partial_{j}\,|\,  i,j=2,3,4,5\rangle\otimes\C[[x_1]]\,|\, \diver(X)=0\},\\
	(\g_0)_{\bar{1}}=& \langle x_id_{1j}+x_jd_{1i}\,|\,  i,j=2,\dots, 5\rangle\otimes\C[[x_1]]\\
	&+\langle fd_{ij}+\frac{1}{2}f' (x_jd_{i1}-x_id_{j1})\,|\, i,j=2,\dots, 5, f\in\C[[x_1]]\rangle.
\end{align*}
We recall the description of $E(1,6)$ from Section \ref{preliminary} as \[E(1,6)=W_1\oplus (\mathfrak{sl}_4\otimes \mathbb C[[t]])\oplus (S^2(\C^4)\otimes{\Omega^1(1)}^{-\frac{1}{2}})\oplus (\displaywedge^2(\C^4)\otimes{\Omega^1(1)}^{-\frac{3}{2}})\]
and define a linear map $\psi:E(1,6)\rightarrow \g_0$ in the following way.
\begin{enumerate}
	\item For all  $f(t)\de_t\in W_1$ we let
	\[\psi(f(t)\de_t)=f(x_1)\de_1-\frac{1}{4}f'(x_1)\sum_{k=2}^5 x_k \de_k;\]
	\item for all $X\otimes f(t)\in \mathfrak{sl}_4 \otimes \C[[t]]$ we let
	\[\psi(X\otimes f(t))= f(x_1) X;
	\]
	\item for all $x_ix_j\otimes f(t)dt\in S^2(\C^4)\otimes{\Omega^1(1)}^{-\frac{1}{2}}$ we let
	\[
	\psi(x_ix_j\otimes f(t)dt)=f(x_1)(x_id_{1j}+x_jd_{1i})=-d(f(x_1)x_ix_j d_1);
	\]
	\item for all $d_{ij}\otimes f(t)dt\in \inlinewedge^2(\C^4)\otimes \Omega^1(1)^{-3/2}$ we let
	\[
	\psi (d_{ij}\otimes f(t)dt)=f(x_1)d_{ij}+\frac{1}{2} f'(x_1)\,(x_i d_{1j}-x_jd_{1i})=\frac{1}{2}d\big(f(x_1)(x_id_j-x_j d_i)\big).
	\]
\end{enumerate}
The following theorem holds (see also \cite{SurThe}).

\begin{theorem}\label{g00}
	The map $\psi: E(1,6)\rightarrow \g_0$ defined above is a Lie superalgebra isomorphism.
\end{theorem}
\begin{proof}
	First note that $\psi$ is indeed an isomorphism of supervector spaces: this follows from the observation that $X\in (\g_0)_{\bar{0}}$ can be written as a 0-divergence vector field $X=\sum_{i=1}^5 f_i \de_i$ with $f_1\in \mathbb C[[x_1]]$.  Therefore $X=(f_1\de_1-\frac{1}{4}f'_1\sum_{k=2}^5 x_k \de_k)+(\frac{1}{4} f'_1\sum_{k=2}^5 x_k \de_k+\sum_{k=2}^5 f_k \de_k)=(f_1\de_1-\frac{1}{4}f'_1\,\sum_{k=2}^5 x_k \de_k)+gY$ for some $0$-divergence vector field $Y\in \mathfrak{sl}_4$  and some $g\in\C[[x_1]]$.
	
	Next we verify that $\psi$ is a homomorphism of Lie superalgebras, i.e.
	\begin{equation}\label{hom}
		[\psi(\alpha),\psi(\beta)]=\psi([\alpha,\beta])
	\end{equation}
	for all $\alpha,\beta\in E(1,6)$. If $\alpha$ and $\beta$ are both even Equation \eqref{hom} can be easily verified.
	
	Let us verify  Equation \eqref{hom} if $\alpha$ is even and $\beta$ is odd. 
	\begin{itemize}
		\item Let $\alpha=f(t)\de_t\in W_1$ and $\beta=x_ix_j\otimes g(t)dt\in S^2(\C^4)\otimes{\Omega^1(1)}^{-\frac{1}{2}}$. We have
		\begin{align*}
			[\psi(f(t)\de_t),\psi(x_ix_j&\otimes g(t)dt)]=[f(x_1) \de_1-\frac{1}{4}f'(x_1)\sum_{k=2}^5 x_k \de_k,\,g(x_1)(x_id_{1j}+x_jd_{1i})]\\
			&=d(f(x_1)g(x_1)(x_id_j+x_jd_i))+\frac{1}{2}d(f'(x_1)g(x_1) x_i x_j d_1)\\
			&= \psi(x_ix_j\otimes \big(\de_t(f(t)g(t))-\frac{1}{2}f'(t)g(t)\big)dt)\\
			&=\psi([f(t)\de_t,x_ix_j\otimes g(t)dt]).
		\end{align*}
		
		\item Let $\alpha=f(t)\de_t\in W_1$ and $\beta=d_{ij}\otimes g(t)dt\in \inlinewedge^2(\C^4)\otimes \Omega^1(1)^{-3/2}$. We have
		\begin{align*}
			[\psi(f(t)\de_t),\psi(d_{ij}&\otimes g(t)dt)]=[f(x_1)\de_1-\frac{1}{4}f'(x_1) \sum_{k=2}^5 x_k \de_k, g(x_1) d_{ij} +\frac{1}{2}g'(x_1)(x_id_{1j}-x_j d_{1i})]\\
			&= \frac{1}{2}d(f(x_1)g'(x_1)(x_id_j-x_j d_i))-\frac{1}{4}d(f'(x_1)g(x_1)(x_id_j-x_j d_i))\\
			&= \psi(d_{ij}\otimes (f(t)g'(t)-\frac{1}{2}f'(t)g(t))dt)\\
			&= \psi([f(t)\de_t,d_{ij}\otimes g(t)dt]).
		\end{align*}
		\item Let $\alpha=x_h \de_k \otimes f(t)\in \mathfrak{sl}_4 \otimes \C[[t]]$ and $\beta=x_ix_j\otimes g(t)dt\in S^2(\C^4)\otimes{\Omega^1(1)}^{-\frac{1}{2}}$. We have
		\begin{align*}
			[\psi(x_h \de_k \otimes f(t)), \psi(x_ix_j &\otimes g(t) dt)]=[f(x_1)x_h \de_k, g(x_1)(x_i d_{1j}+x_j d_{1i})]\\
			&= \delta _{kj}(-d(f(x_1) g(x_1)x_hx_id_1)+\delta _{ki}(-d(f(x_1) g(x_1)x_hx_jd_1)\\
			&= \delta _{kj}\psi(x_hx_i\otimes f(t)g(t) dt)+\delta_{ki}\psi(x_hx_j\otimes f(t)g(t)dt)\\
			&=\psi([x_h \de_k\otimes f(t),x_ix_j \otimes g(t) dt]).
		\end{align*}
		The case $\alpha=x_h \de_h-x_k \de_k$ is similar and hence is omitted.
		\item 
		Let $\alpha=x_h \de_k \otimes f(t)\in \mathfrak{sl}_4 \otimes \C[[t]]$ and $\beta=d_{ij}\otimes g(t)dt\in \inlinewedge^2(\C^4)\otimes \Omega^1(1)^{-3/2}$. We can assume without loss of generality that $k\neq j$. We have
		\begin{align*}
			[\psi(x_h \de_k \otimes f(t)), \psi(d_{ij}&\otimes g(t)dt)]=[f(x_1)x_h \de_k,g(x_1) d_{ij} +\frac{1}{2}g'(x_1)(x_id_{1j}-x_j d_{1i})]\\
			&= \delta_{ik}(d(f(x_1)g(x_1)x_h d_j)+\frac{1}{2}d(f(x_1)g'(x_1)x_h x_j d_1))\\
			&=\delta_{ik}(f(x_1)g(x_1)d_{hj}+\frac{1}{2}\de_1(f(x_1)g(x_1)) (x_hd_{1j}-x_j d_{1h})\\
			&\hspace{5mm}+\frac{1}{2} f'(x_1) g(x_1)(x_h d_{1j}+x_j d_{1h}) )\\
			&= \delta_{ik}(\psi(d_{hj}\otimes f(t)g(t) dt)+\frac{1}{2} \psi(x_hx_j\otimes f'(t) g(t) dt)).
		\end{align*}
		On the other hand
		\begin{align*}
			[x_h \de_k \otimes f(t), d_{ij}\otimes g(t) dt]&=\delta_{ik}(d_{hj}\otimes f(t) g(t)dt+ {\textstyle \int} x_h d_j \otimes f'(t) g(t) dt)\\
			&= \delta_{ik}(d_{hj}\otimes f(t) g(t)dt+\frac{1}{2}x_hx_j \otimes  f'(t) g(t) dt).
		\end{align*}
		The case $\alpha=x_h \de_h-x_k \de_k$ is similar and hence is omitted.
	\end{itemize}
	
	Finally, we can verify Equation \eqref{hom} in the case where both $\alpha$ and $\beta$ are odd.
	\begin{itemize}
		\item If $\alpha,\beta\in S^2(\C^4)\otimes{\Omega^1(1)}^{-\frac{1}{2}}$ we can observe that $	[\psi(\alpha),\psi(\beta)]=\psi([\alpha,\beta])=0$.
		\item  Let $\alpha=x_hx_k\otimes f(t) dt\in  S^2(\C^4)\otimes{\Omega^1(1)}^{-\frac{1}{2}}$ and $\beta= d_{ij}\otimes g(t) dt \in  \inlinewedge^2(\C^4)\otimes \Omega^1(1)^{-3/2}$. We first assume that $i,j,h,k$ are distinct and that $\varepsilon_{1ijhk}=1$. We have
		\begin{align*}
			[\psi(x_hx_k\otimes f(t) dt),\psi(d_{ij}&\otimes g(t) dt)]=[f(x_1)(x_hd_{1k}+x_k d_{1h}),g(x_1)d_{ij}+\frac{1}{2} g'(x_1)(x_i d_{1j}-x_j d_{1i})]\\
			&=f(x_1)g(x_1)(-x_h \de_h+x_k \de_k)\\
			&=\psi((-x_h \de_h+x_k \de_k)\otimes f(t)g(t)).
		\end{align*}
		On the other hand 
		\[
		[x_hx_k\otimes f(t)dt,d_{ij}\otimes g(t)dt]=(x_kd_h+x_hd_k)\wedge d_{ij}\otimes f(t) g(t)=(-x_k\de_k+x_h \de_h)\otimes f(t) g(t)
		\]
		If $i,j,h,k$ are not distinct the computation is similar and simpler and is therefore omitted.
		\item Let $\alpha=d_{ij}\otimes f(t) dt, \beta=d_{hk}\otimes g(t)dt \in \inlinewedge(\mathbb C^4)\otimes \Omega^1(1)^{-\frac{3}{2}}$. We assume $\epsilon(1ijhk)=1$.  We have
		\begin{align*}
			[\psi(d_{ij}\otimes& f(t) dt),\psi(d_{hk}\otimes g(t)dt)]=\\&=[f(x_1)d_{ij}+\frac{1}{2}f'(x_1)(x_id_{1j}-x_j d_{1i}),g(x_1) d_{hk}+\frac{1}{2}g'(x_1)(x_h d_{1k}-x_k d_{1h})]\\
			&=f(x_1)g(x_1)\de_1-\frac{1}{2}f(x_1)g'(x_1)(x_h\de_h+x_k \de_k)-\frac{1}{2}f'(x_1)g(x_1)(x_i\de_i+x_j \de_j)\\
			&=f(x_1)g(x_1)\de_1 -\frac{1}{4}\de_1(f(x_1)g(x_1))\sum_{r=2}^5 x_r\de_r\\
			&\,\,+\frac{1}{4}(f(x_1)g'(x_1)-f'(x_1)g(x_1))(x_i\de_i+x_j \de_j-x_h \de_h-x_k \de_k)\\
			&=\psi(f(t)g(t)\de_t)+\psi((x_i\de_i+x_j \de_j-x_h \de_h-x_k \de_k)\otimes \frac{1}{4}(f(t)g'(t)-f'(t)g(t))).
		\end{align*}
		On the other hand
		\begin{align*}
			[d_{ij}\otimes f(t) dt,d_{hk}&\otimes g(t)dt]=\\
			&= f(t) g(t) \de_t+\frac{1}{2}\big(\textstyle{\int} d_{ij}\wedge d_{hk}-d_{ij}\wedge \textstyle{\int} d_{hk}\big)\otimes (f'(t)g(t)-f(t)g'(t))\\
			&= f(t) g(t) \de_t+\frac{1}{4}(x_i\de_i+x_j\de_j-x_h\de_h -x_k \de_k )\otimes (f(t)g'(t)-f'(t)g(t)).
		\end{align*}
	\end{itemize} 
\end{proof}
\color{black}

In the embedding of $L=E(1,6)$ into $\g=E(5,10)$ described in Theorem \ref{g00} the grading element  of $L$ in its principal grading is $z=2x_1\partial_{1}-\frac{1}{2}\sum_{i=2}^5x_i\partial_{i}$: this corresponds to letting
\[\deg x_1=-\deg\de_1=2,\,\deg x_i=-\deg\de_i=1\,\,{\mbox{for}}\,\,i=2,3,4,5,\,\,{\mbox{and}}\,\,\deg d=-\frac{3}{2};\]
 (it is the $\Z$-grading of type
$(2,1,1,1,1)$ of $E(5,10)$). It follows that inside $E(5,10)$ the non-positive part $L_{\leq 0}$ of $E(1,6)$, with respect to its principal grading, is the following:

\[L_{<0}=\langle \partial_{x_1}, d_{ij}\,|\, i,j=2,3,4,5, i<j\rangle,\]
\[L_0=\langle z, x_i\partial_{j}, x_i\partial_{i}-x_{i+1}\partial_{i+1}\, |\, i,j=2,3,4, i\neq j\rangle.\]

Since $\g_0$ is isomorphic to $E(1,6)$, each $\g_j$ is an $E(1,6)$-module.
We have:
\[\g_{-1}=\langle \partial_{i}\,|\, i=2, \dots, 5\rangle\otimes\C[[x_1]]+\langle d_{1j}\,|\,j=2,\dots,5\rangle\otimes\C[[x_1]],\]
and, for $r\geq 1$,
\begin{align*}
(\g_r)_{\bar{0}} = \{ & X\in \langle p\partial_{1},
q\partial_{i}\rangle\otimes \C[[x_1]]\,|\, i=2, \dots,5, p,q\in\C[x_2,\dots, x_5],
 \deg(p)=r,\\
 &\deg(q)=r+1, div(X)=0\};
 \end{align*}
\begin{align*} 
(\g_r)_{\bar{1}}= \{ &\omega\in \langle pd_{ij},
qd_{1j}\rangle\otimes \C[[x_1]]\,|\, i,j=2, \dots,5, p,q\in\C[x_2,\dots, x_5], \deg(p)=r, \\
& \deg(q)=r+1, d\omega=0\}.
\end{align*}
 
\begin{proposition}\label{generated}
The $\g_0$-module $\g_r$ is generated by the vector $v_r=x_2^r\partial_{1}$ if $r\geq 0$ and by the vector $v_{-1}=\partial_{5}$ if $r=-1$.
\end{proposition}
\begin{proof}
Since $\g$ is a simple $\Z$-graded Lie superalgebra of depth 1, the $\g_0$-module $\g_{-1}$ is irreducible, hence is generated by $v_{-1}$, and we may assume that $r\geq 0$. 
We have
\[(\g_r)_{\bar{0}} =\langle x_1^a\de_{2}(q)\partial_{1}-ax_1^{a-1}q\partial_{2};
x_1^bY
\,|\, a, b\in\Z_+, \deg(q)=r+1, Y\in (S_4)_r\rangle,\]

\medskip

\noindent
where $(S_4)_r$ denotes the subspace of zero-divergence vector fields in the variables $x_2, \dots, x_5$, of principal degree $r$ (i.e., $\deg x_i=-\deg
\partial_{i}=1$).
The  vector fields of the form $x_1^bY$, such that $Y\in (S_4)_r$, span the   
$(\g_0)_{\bar{0}}$-submodule $(S_4)_r\otimes \C[[x_1]]$ of $(\g_r)_{\bar{0}}$.
 
We have:
\[[x_1^kx_3\partial_{4}, v_r]=-kx_1^{k-1}x_2^rx_3\partial_{4},\]
hence $(S_4)_r\otimes \C[[x_1]]$ is contained in the $\g_{\bar{0}}$-submodule generated by $v_r$, since $(S_4)_r$ is an irreducible $\mathfrak{sl}_4$-module.

The quotient of $(\g_r)_{\bar{0}}$ by the submodule $(S_4)_r\otimes \C[[x_1]]$  is isomorphic to
$S^r(4)\otimes W_1$: two vector fields $\sum_{i=1}^5 f_i\partial_i$ and $\sum_{i=1}^5 g_i\partial_i$ lie in the same class if and only if $f_1=g_1$.

We have
\[ [x_1^kx_3\partial_{2}, v_r]=rx_1^{k}x_2^{r-1}x_3\partial_{1}-
kx_1^{k-1}x_2^rx_3\partial_{2},\]
hence, by the irreducibility of $S^r(\C^4)$ as an $\mathfrak{sl}_4$-module, $(\g_r)_{\bar{0}}$ is
generated by $v_r$.

The odd part of $\g_r$ is given by:
$$(\g_r)_{\bar{1}}=
\langle x_1^ad\sigma-d(x_1^a)\wedge\sigma, x_1^bdx_1\wedge\tau\,|\,
 a,b\in\Z_+,
\sigma\in\Omega^1_{r+1}(4), \tau\in\Omega^1_{r+1}(4)_{cl}\rangle,$$
where $\Omega^1_r(4)$ and ${\Omega^1}_{r+1}(4)_{cl}$ denote the space of differential 1-forms of degree $r$ and closed differential 1-forms of degree $r+1$, respectively, in the four variables 
$x_2, \dots, x_5$ ($\deg x_i=1=\deg dx_i$).

The subspace  consisting of closed differential 2-forms $x_1^bdx_1\wedge \tau$, such that $\tau\in \Omega^1_{r+1}(4)_{cl}$, is  a
$(\g_0)_{\bar{0}}$-submodule of $(\g_r)_{\bar{1}}$. 
We have:
\[[x_1d_{23}+x_2d_{13}, v_r]=
-(r+1)x_2^rd_{23}.\]
By the irreducibility of $\Omega^2_r(4)_{cl}$ as an $\mathfrak{sl}_4$-module, the form
$x_2^rd_{23}$ generates the whole $\Omega^2_r(4)_{cl}$ as an $\mathfrak{sl}_4$-module.
Besides, for $x_1^aY\in \C[x_1]\otimes \mathfrak{sl}_4$ and $\omega\in\Omega^2_r(4)_{cl}$,
\[[x_1^aY,\omega]=d(x_1^ai_Y(\omega))=ax_1^{a-1}dx_1\wedge i_Y(\omega)+
x_1^ad(i_Y(\omega)),\]
hence we obtain all elements in $(\g_r)_{\bar{1}}$ since  $\Omega^1_{r+1}(4)$ is spanned by contractions of  closed 2-forms of degree $r$ along  vector fields in $\mathfrak{sl}_4$. Indeed, let $\sigma=gdx_j$ for some monomial $g$ of degree $r+1$ in the variables $x_2, \dots, x_5$ and some $j\in\{2,\dots, 5\}$. Let  $k\in\{2,\dots, 5\}$ be such that $x_k$ divides $g$ and let $h\neq k,j$. Take $f\in \C[x_2, \dots, x_5]$, a monomial of degree $r+1$, such that $x_k\de_{h}f=g$ and $\omega=d(fdx_j)=\sum_i\de_{x_i}fd_{ij}$. Then
\[i_{x_k\de_{h}}(\omega)=x_k\de_{h}fdx_j=gdx_j.\]


\end{proof}

Each of the subspaces $\g_j$ carries a $\Z$-grading by finite-dimensional subspaces induced by the principal grading of $E(5,10)$, i.e., the grading of type $(2,2,2,2,2)$. Recall that in this grading $\deg d=-5/2$, therefore we have the induced $\Z$-grading
\[\g_r=\prod_{k\geq 2r-2}\g_{r,k}.\]

Each $E(1,6)$-module $\g_r$ is a linearly compact space, hence the dual modules are discrete spaces, and we have
\[\g_r^*=\bigoplus_{k\geq 2r-2}\g_{r,k}^*,\]
where all $\g_{r,k}$ and $\g_{r,k}^*$ are $\g_0$-modules. In fact, all $E(1,6)$-modules $\g_r^*$ are objects in the category $\mathcal{P}(L,L_{\geq 0})$
of continuous $\Z$-graded modules with discrete topology over a $\Z$-graded linearly compact Lie superalgebra $L$, finitely generated as $\mathcal{U}(L_{<0})$-modules (\cite{CCK0}).

\medskip

We denote by $M(\lambda)$ the $E(1,6)$-Verma module induced by an irreducible finite-dimensio\-nal $L_0$-module of highest weight $\lambda=(a,b,c;d)$ where
$a, b, c\in\Z_+$ and $d\in\C$ denote the eigenvalues of the elements $x_2\de_{2}-
x_3\de_{3}$, $x_3\de_{3}-
x_4\de_{4}$, $x_4\de_{4}-
x_5\de_{5}$ and $z$, respectively, extended by zero to $L_{>0}$.
Besides, we denote by $I(\lambda)$ the irreducible quotient of $M(\lambda)$ by its unique maximal submodule.

\begin{theorem}\label{3.5} The $E(1,6)$-modules $\g_j^*$ are irreducible  for every $j\geq -1$, and
$\g_{-1}^*\cong I(1,0,0; -\frac{1}{2})$, 
$\g_{r}^*\cong I(0,0,r; \frac{r}{2}+2)$, for $r\geq 0$.
\end{theorem}
\begin{proof}
The vector $v_{-1}=\partial_{5}$ in $\g_{-1}$ is a  highest weight vector of weight $(0,0,1; \frac{1}{2})$, which is annihilated by the negative part of $E(1,6)$. Therefore $\g_{-1}^*$, being irreducible, is isomorphic to the unique irreducible quotient $I(1,0,0; -\frac{1}{2})$  of the Verma module 
$M(1,0,0; -\frac{1}{2})$.

Similarly, the vector $v_0=\partial_{1}$ in $\g_{0}$ is a  highest weight vector of weight $(0,0,0; -2)$, which is annihilated by the negative part of the simple linearly compact Lie superalgebra $E(1,6)$. Therefore $\g_{0}^*$ is isomorphic to the unique irreducible quotient $I(0,0,0; 2)$  of the Verma module 
$M(0,0,0; 2)$. Hence we may assume $r\geq 1$.

The vector $v_r=x_2^r\partial_{1}$ in $\g_{r}$ is a  highest weight vector of weight $(r,0,0; -\frac{r}{2}-2)$, which is annihilated by the negative part of $E(1,6)$ and generates $\g_r$ as a $\g_0$-module by Proposition \ref{generated}. It follows that
$\g_r^*$ is generated by a highest weight vector $w$ of weight $(0,0,r; \frac{r}{2}+2)$
hence we can construct a morphism of $\g_0$-modules:
\[\psi: M(0,0,r; \frac{r}{2}+2) \rightarrow \g_r^*\]
sending the generating vector of $M(0,0,r; \frac{r}{2}+2)$ to $w$.
It follows that $\g_r^*$ is isomorphic to a quotient of the Verma module 
$M(0,0,r; \frac{r}{2}+2)$. Notice that $\psi$ is a graded morphism sending
$M(0,0,r; \frac{r}{2}+2)_k$ to $(\g_{r, 2r-2+k})^*$.

Now let $W$ be a proper submodule of $\g_r^*$. Then $\psi^{-1}(W)$ is a proper submodule of $M(0,0,r; \frac{r}{2}+2)$, hence it contains the unique singular vector $v$ of $M(0,0,r; \frac{r}{2}+2)$ of positive degree (see \cite[Theorem 4.1]{BKL1} and \cite[Section 3]{B}). This vector has degree 1, hence $\psi(v)$ is
a singular vector of degree 1 in $(\g_{r,2r-1})^*$. But this is impossible since $\g_{r,2r-1}$ does not contain non-zero vectors annihilated by the negative part of $E(1,6)$. Indeed,
\[\g_{r,2r-1}=\{\omega\in\langle pd_{ij}\rangle \,|\, \deg(p)=r, d\omega =0\},\]
hence if $\psi(v)=\sum_{2\leq i<j\leq 5}P_{ij}d_{ij}$ then 
$d_{kt}(\psi(v))=\sum_{2\leq i<j\leq 5}\epsilon_{[ktij]}P_{ij}\partial_{[ktij]}$, hence $d_{kt}(\psi(v))=0$ for every $k,t$ if and only if $P_{ij}=0$ for every $i,j$.

It follows that $\g_r^*$ is irreducible and hence isomorphic to 
$I(0,0,r; \frac{r}{2}+2)$.

\end{proof}

\section{Embedding the Lie superalgebra $E(1,6)$ into $E(4,4)$}
Consider the linearly compact Lie superalgebra $K(1,6)=\C[[t]]\otimes \bigwedge(\xi_1,\xi_2,\xi_3,\eta_1,\eta_2,\eta_3)$ with bracket given by  
\[
[f,g]=(2-E)f \frac{\de g}{\de t}- \frac{\de f}{\de t}(2-E)g+(-1)^{p(f)}\sum_{i=1}^3 \Big(\frac{\de f}{\de \xi_i}\frac {\de g}{\de \eta_ {i}}+\frac{\de f}{\de  \eta_{i}}\frac{\de g}{\de \xi_ i}\Big),
\]
where $E=\sum_{i=1}^3 (\xi_i \frac{\de}{\de \xi_i}+\eta_{i} \frac{\de}{\de  \eta_{i}})$.
We shall need also a slightly different description of $K(1,6)$. Namely, letting for all $j=1,2,3$,
\[
\rho_j=\frac{1}{\sqrt 2} (\xi_j+\eta_j),\,\,\rho_{j+3}=\frac{1}{\sqrt {-2}} (\xi_j-\eta_j),
\]
i.e.,
\[
\xi_j=\frac{1}{\sqrt 2} (\rho_j+\sqrt{-1}\rho_{j+3}),\,\,\eta_j=\frac{1}{\sqrt 2} (\rho_j-\sqrt{-1}\rho_{j+3}),
\]
we have $K(1,6)=\C[[t]]\otimes \bigwedge(\rho_1,\ldots,\rho_6)$ with bracket given by  
\[
[f,g]=(2-E)f \frac{\de g}{\de t}- \frac{\de f}{\de t}(2-E)g+(-1)^{p(f)}\sum_{i=1}^6 \frac{\de f}{\de \rho_i}\frac {\de g}{\de \rho_ {i}},
\]
and one can check that we also have $E=\sum_{i=1}^6 \rho_i \frac{\de}{\de \rho_i}$.
In our arguments we make use of both descriptions of $K(1,6)$ since the bracket involves simpler computations if one uses the $\rho_i$'s and, on the other hand, we have a simpler decomposition of $K(1,6)$ as a graded Lie superalgebra if one uses the $\xi_i$'s and the $\eta_i$'s.
\begin{remark}
	We have $\rho_j\rho_{j+3}=\sqrt{-1}\xi_j\eta_j$ for all $j=1,2,3$.
\end{remark}
If $I=(i_1,\ldots,i_k)$ is a sequence with distinct entries in $\{1,2,3,4,5,6\}$ we let $\rho_I=\rho_{i_1}\cdots \rho_{i_k}$ and $\rho_I^*$ be the unique monomial such that $\rho_I \rho_I^*=\rho_1\rho_2\cdots \rho_6$. The $*$-operator is extended to $\bigwedge(\rho_1,\ldots,\rho_6)$ by linearity. 

We recall that a $\mathbb Z$-grading of $K(1,6)$ can be defined by assigning an integer degree to the variables $t,\xi_i,\eta_i$ ($i=1,2,3$), such that $\deg t=\deg \xi_i+\deg \eta_i$ for all $i=1,2,3$. We shall call $(\deg t |\deg \xi_1,\deg \xi_2, \deg \xi_3, \deg \eta_1,\deg \eta_2, \deg \eta_3)$ the type of the grading. Then the degree of a monomial is equal to the sum of degrees of factors minus 1.
We consider on $K(1,6)$ the grading of type $(1|0,1,1,1,0,0)$. Although this is isomorphic to the more standard grading $(1|1,1,1,0,0,0)$ it will induce a different grading on the subalgebra $E(1,6)$ (see \cite[Remark 6.1]{CaK}) that we are going to discuss. It is sometimes convenient to rename $\eta_1=\xi_4$ and $\xi_1=\eta_4$ so that 

\[
K(1,6)=\C[[t]]\otimes \bigwedge(\xi_2,\xi_3,\xi_4,\eta_2,\eta_3,\eta_4)
\]
and all the $\xi_i's$ (resp.\ the $\eta_i$'s) have the same degree (1 and 0 respectively). In particular, if $2\leq i_1<\cdots<i_k\leq 4$ and $2\leq j_1<\cdots<j_h\leq 4$ we have $\deg (t^n \xi_{i_1}\cdots \xi_{i_k}\eta_{j_1}\cdots \eta_{j_h})=n+k-1$. We observe that elements $\rho_j$ are not homogeneous with respect to this grading. 

For $i\in \{1, 2, 3, 4\}$ it is also convenient to set $\xi_{\bar i}=\eta_i$
(the bar here is just a symbol).

If $I=(i_1,\ldots,i_k)$ is a sequence with distinct entries in either $\{1,2,3,\bar 1,\bar 2,\bar3\}$ or in $\{2,3,4,\bar 2, \bar 3, \bar 4\}$ we let $|I|=k$, $\xi_I=\xi_{i_1}\cdots \xi_{i_k}$ and $\bar I=(\bar{i_1},\ldots,\bar{i_k})$ (where $\bar{\bar i}=i$). If $f\in K(1,6)$ is a non-zero monomial $f=at^n \xi_I$ we also let $|f|=|I|$. 
 Moreover, we let $\xi_I^\#$ be the unique monomial such that $\xi_{\bar I} \xi_I^\#=\xi_2 \xi_3 \xi_4 \eta_{2} \eta_{3} \eta_{4}=-\xi_1 \xi_2 \xi_3 \eta_{1} \eta_{2} \eta_{3}$. For example, if $I=(2,\bar 4)$ then $\bar I=(\bar 2,4)$ and $\xi_I^\#=-\xi_2 \xi_3 \eta_{3} \eta_{4}$. The $\#$-operator is also extended to $\bigwedge(\xi_1,\xi_2,\xi_3,\eta_1,\eta_2,\eta_3)$ by linearity.

\begin{lemma}\label{diesis} For all $X\in \bigwedge(\xi_1,\xi_2,\xi_3,\eta_1,\eta_2,\eta_3)=\bigwedge(\rho_1,\rho_2,\rho_3,\rho_4,\rho_5,\rho_6)$
we have \[X^*=\sqrt{-1} X^{\#}.\]
\end{lemma}
\begin{proof} It is enough to show the result for $X=\rho_I$ for some sequence $I$ with distinct entries in $\{1,2,3,4,5,6\}$. We observe that there esists a partition $I_1,I_2,I_3,I_4$ of $\{1,2,3\}$ 
	 such that (up to a sign)
	\[
	\rho_I=\prod_{j\in I_1}\rho_j \rho_{j+3} \prod_{j\in I_2} \rho_j \prod_{j\in I_3}\rho_{j+3}.
	\]
	Then 
	\[
	\rho_I^*= \epsilon \prod_{j\in I_2} \rho_{j+3} \prod_{j\in I_3}\rho_{j}\prod_{j\in I_4}\rho_j \rho_{j+3},
	\]
	where $\epsilon=\pm 1$ is such that 
	\[
	\epsilon \prod_{j\in I_1}\rho_j \rho_{j+3} \prod_{j\in I_2} \rho_j \prod_{j\in I_3}\rho_{j+3}\prod_{j\in I_2} \rho_{j+3} \prod_{j\in I_3}\rho_{j}\prod_{j\in I_4}\rho_j \rho_{j+3}=\rho_1 \rho_2 \rho_3 \rho_4 \rho_5 \rho_6.
	\]
	
	Next we observe that if $\xi_J \eta_j\neq 0$, $\xi_J \xi_j\neq 0$ and $(\xi_J\eta_j)^\#=\xi_K \eta_j$ for some sequences $J,K$ in $\{1,2,3,\bar 1, \bar 2, \bar 3\}$ and some index $j\in \{1,2,3\}$, then $(\xi_J\xi_j)^\#=-\xi_K \xi_j$.
	Moreover we observe that
	\[
	\Big(\prod_{j\in I_1} \eta_j \xi_j \prod_{j\in I_2} \eta_j \prod_{j\in I_3} \xi_j \Big)^\#=\epsilon' \prod_{j\in I_2}{\eta_j} \prod_{j\in I_3}\xi_j \prod_{j\in I_4} \xi_j \eta_j,
	\]
	where the sign $\epsilon=\pm1$ is such that 
	\[
	\epsilon'\prod_{j\in I_1} \xi_j \eta_j \prod_{j\in I_2} \xi_j \prod_{j\in I_3} \eta_j\prod_{j\in I_2}{\eta_j} \prod_{j\in I_3}\xi_j \prod_{j\in I_4} \xi_j \eta_j=\epsilon' \xi_2\xi_3\xi_4\eta_2\eta_3\eta_4=-\epsilon' \xi_1\xi_2\xi_3\eta_1\eta_2\eta_3.
	\]
	In particular we can observe that $\epsilon'=-\epsilon$. From these observations we can deduce that
	\begin{align*}\rho_I^\#&=\Big(\prod_{j\in I_1}\sqrt{-1}\xi_j \eta_j \prod_{j\in I_2} \frac{1}{\sqrt 2}(\xi_j+\eta_{j})\prod_{j\in I_3} \frac{1}{\sqrt {-2} }(\xi_j-\eta_j)\Big)^\#\\
		&=-\epsilon_I (-\sqrt{-1})^{|I_1|}\prod_{j\in I_2}\frac{1}{\sqrt 2} (\eta_j-\xi_j) \prod_{j\in I_3} \frac{1}{\sqrt {-2}}(\xi_j+\eta_j) \prod_{j\in I_4} \xi_j \eta_j\\
		&=-\epsilon_I (-\sqrt{-1})^{|I_1|+|I_2|+|I_3|+|I_4|} \prod _{j\in I_2} \rho_{j+3} \prod_{j\in I_3} \rho_j \prod_{j\in I_4} \rho_j \rho_{j+3}\\
		&= - \sqrt{-1} \rho_I^*,
	\end{align*}
	since $|I_1|+|I_2|+|I_3|+|I_4|=3$.
	
	%
\end{proof}
We let $A:K(1,6)\rightarrow K(1,6)$ be the linear operator defined by
\[
A(t^n \xi_I)=(-1)^{|I|(|I|+1)/2}\de _t^{3-|I|}t^n \xi_I^\#,
\]
where $\de_t^{-1}=\int_0^t$ (homogeneous integration with respect to $t$), $I$ is any sequence with distinct coefficients in $\{2,3,4,\bar 2,\bar 3,\bar 4\}$ and $n\geq 0$. 

By Lemma \ref{diesis}, if $I$ is a sequence with coefficients in $\{1,2,3,4,5,6\}$ we also have 
\[A(t^n \rho_I)=-\sqrt{-1}(-1)^{|I|(|I|+1)/2}\de_t^{3-|I|}t^n \rho_I^*.\]
Moreover, since $(\rho_I^*)^*=(-1)^{|I|}\rho_I$ it follows immediately that  $A(t^n \rho_I)\neq 0$ then $A(A(t^n \rho_I))=t^n \rho_I$.
Finally, for all $f \in K(1,6)$ we let 
\[\iota (f)=f+A(f).\]

For all $k\geq 0$ we let $L_k=\langle t^n \rho_I:\, n+|I|=k\rangle$. We clearly have $K(1,6)=\prod_{k\geq 0} L_k$ as vector spaces and we observe that $A$ is injective on $L_k$ for $k\geq 3$ and is identically zero on $L_0,L_1,L_2$.

Cheng and Kac \cite[Remark 5.3.2]{CK} observed (without a proof) that elements of the form $\iota(f)$ span a Lie subalgebra of $K(1,6)$ isomorphic to $E(1,6)$. This fact is a consequence of the following results that will also be useful in the sequel.

\begin{definition} Let $I,J$ be sequences in $\{1,2,3,4,5,6\}$ and $n,m\geq 0$. Let $f=t^n\rho_I$ and $g=t^m \rho_J$ be monomials in $K(1,6)$.  We say that $(f,g)$ is an exceptional pair if the following conditions hold:
	\begin{itemize}
\item $n=m=0$;
\item $|I|+|J|\geq 4$;
\item $|I|\leq 1$ or $|J|\leq 1$;
\item $\rho_I \rho_J\neq 0\in \bigwedge(\rho_1,\ldots,\rho_6)$.
\end{itemize}
\end{definition}
\begin{lemma}\label{propertiesofA} Let $f=t^n\rho_I$ and $g=t^m \rho_J$ be monomials in $K(1,6)$. If $(f,g)$ is not an exceptional pair then we have
	\[
	A\Bigl([f,g]+[A(f),A(g)]\Bigr)=[A(f),g]+[f,A(g)].
	\]
If $(f,g)$ is an exceptional pair then $A\big([A(f),g]+[f,A(g)] \big)=[f,g]+[A(f),A(g)]=0$.
\end{lemma}
\begin{proof}
		Let $f=t^n \rho_I$ and $g=t^m \rho_J$ and for notational convenience let $d=|I|$ and $e=|J|$. We first prove the statement if the following two conditions are satisfied: 
	\begin{enumerate}
		\item[a)]$d\leq e\leq 3$;
		\item[b)] if $e=3$ then $|I\cap J|\geq |I\cap J^c|$.
	\end{enumerate} 
	Note that conditions a) and b) above imply that $(f,g)$ is not an exceptional pair. 
	
	We also observe that conditions a) and b) imply that $[A(f),A(g)]=0$ since we necessarily have $|I^c\cap J^c|\geq 2$, where $I^c$ denotes the complement of $I$ in $\{1,2,3,4,5,6\}$.
	We will show that if a) and b) are satisfied then
	\begin{equation}\label{Afg}
		\sqrt{-1} A([f,g])=\sqrt{-1}[A(f),g]+\sqrt{-1}[f,A(g)]
	\end{equation}
	by a case by case analysis.
	We have
	\begin{align*}
		\sqrt{-1}A([f,g])=&(-1)^{\frac{(d+e)(d+e+1)}{2}}\big((2-d)\de_t^{3-d-e}(t^n\de_t t^m)-(2-e)\de_t^{3-d-e}(\de_t t^n t^m)\big)(\rho_I\rho_J)^*\\
		&+(-1)^{\frac{(d+e-2)(d+e-1)}{2}}(-1)^d \de_t^{5-d-e} t^{n+m} \sum_j(\de_{\rho_j}\rho_I \de_{\rho_j}\rho_J)^*,
	\end{align*}
	\begin{align*}
		\sqrt{-1}[A(f),g])=&(-1)^{\frac{d(d+1)}{2}}\big((d-4)\de_t^{3-d}t^n\de_t t^m-(2-e)\de_t^{4-d} t^n \,t^m\big)\rho_I^*\rho_J\\
		&+(-1)^{\frac{d(d+1)}{2}}(-1)^d \de_t^{3-d} t^n\,t^m \sum_j\de_{\rho_j}\rho_I^* \de_{\rho_j}\rho_J,
	\end{align*}
	and
	\begin{align*}
		\sqrt{-1}[f,A(g)])=&(-1)^{\frac{e(e+1)}{2}}\big((2-d)t^n\de_t^{4-e} t^m-(e-4)\de_t t^n \de_t^{3-e}t^m\big)\rho_I\rho_J^*\\
		&+(-1)^{\frac{e(e+1)}{2}}(-1)^d  t^n \de_t^{3-e}t^m \sum_j\de_{\rho_j}\rho_I \de_{\rho_j}\rho_J^*.
	\end{align*}
	
	\begin{itemize}
		\item $I=J=\emptyset$. Equation \eqref{Afg} becomes
		\[
		(2\de_t^3(t^n\de_tt^m)-2\de_t^3(\de_tt^nt^m))\,1^*=(-4\de_t^3 t^n \de_t t^m-2\de_t^4t^n t^m+2t^n\de_t^4 t^m+4\de_t t^n \de_t^3 t^m)\,1^*
		\]
		and \eqref{Afg} is satisfied recalling that $\de_t^k(fg)=\sum_{h=0}^k\binom{k}{h} \de_t^h f\de_t^{k-h}g$.
		\item $I=\emptyset$, $J=(k)$. Equation \eqref{Afg} becomes
		\[
		-(2\de_t^2(t^n \de_t t^m)-\de_t^2(\de_t t^n t^m))\rho_k^*=\de_t^3t^n t^m \de_{\rho_k} 1^*+( -2t^n \de_t^3 t^m -2\de_t t^n \de_t^2 t^m)\rho_k^*
		\]
		and it is satisfied with a similar argument observing that $\de_{\rho_k}1^*=\rho_k^*$.
		\item $I=(h)$ and $J=(k)$ with $h\neq k$. Equation \eqref{Afg} becomes
		\[
		-(\de_t(t^n \de_t t^m)-\de_t(\de_tt^n t^m))\rho_{(h,k)}^*=\de_t^2t^n t^m \de_{\rho_k}\rho_h^*+t^n\de_t^2 t^m \de_{\rho_h}\rho_k^*
		\]
		which is satisfied since $\de_{\rho_k}\rho_h^*=-\de_{\rho_h}\rho_k^*=\rho_{(h,k)}^*$.
		\item $I=J=(k)$. Equation \eqref{Afg} becomes
		\[
		-\de_t^3 t^{n+m}\, 1^*=-(-3\de_t^2 t^n \de_tt^m-\de_t^3 t^n t^m)(-1^*)-(t^n \de_t^3 t^m+3\de_tt^n \de_t^2 t^m)1^*.
		\]
		\item $I=(k)$ and $J=(k,h)$.  Equation \eqref{Afg} becomes
		\[
		\de_t^2 t^{n+m} \rho_h^*=\de_t^2 t^n t^m \de_{\rho_h}\rho_k^*(-\rho_k)-(t^n \de_t^2 t^m+2\de_t t^n \de_t t^m)\rho_k \rho_{(k,h)}^*.
		\]
		Observing that $(\de_{\rho_h} \rho_k^*) \rho_k=-\rho_h^*$ and $\rho_k\rho_{(k,h)}^*=-\rho_h^*$ Equation \eqref{Afg} follows.
		\item $I=(k,h)$, $J=(k,l)$, with $h\neq l$. Equation \eqref{Afg} becomes
		\[
		-\de_t t^{n+m} \rho_{(h,l)}^*=-\de_t t^n t^m \de_{\rho_l}\rho_{k,h}^*(-\rho_k)-t^n \de_t t^m (-\rho_k) \de_{\rho_h}\rho_{k,l}^*.
		\]
		We observe here that $\rho_k \de_{\rho_h} \rho_{(k,l)}^*=\de_{\rho_l}\rho_{(k,h)}^* \rho_k=-\rho_{(h,l)}^*$ to deduce that Equation \eqref{Afg} holds.
		\item $I=J$ with $d=2$. Equation \eqref{Afg} becomes
		\[
		0=-(-2\de_t t^n \de_t t^m) \rho_I^* \rho_I-( 2 \de_t t^n \de_t t^m) \rho_I \rho_I^*,
		\]
		which is satisfied since $\rho_I^* \rho_I= \rho_I \rho_I^*$.
		\item $I=J$ with $d=3$
		\[
		0=(-t^n \de_t t^m+\de_t t^n t^m) \rho_I^* \rho_I+(-t^n \de_t t^m + \de_t t^n  t^m) \rho_I \rho_I^*
		\]
		which is satisfied since $\rho_I^* \rho_I= -\rho_I \rho_I^*$.
		\item $I=(h,k)$, $J=(h,k,l)$. Equation \eqref{Afg} becomes
		\[
		0=-\de_t t^n t^m \de_{\rho_l}\rho_{(h,k)}^*\rho_{(h,k)}+\de_t t^n t^m \rho_{(h,k)}\rho_{(h,k,l)}^*
		\]
		which is satisfied since $\de_{\rho_l} \rho^*_{(h,k)}=\rho_{(h,k,l)}^*$.
		\item $I=(h,k,j)$, $J=(h,k,l)$ with $l\neq j$. Equation \eqref{Afg} becomes
		\[
		0=- t^n t^m \de_{\rho_l}\rho_{(h,k,j)}^* \rho_{(h,k)}-t^n t^m \rho_{(h,k)}\de_{\rho_j}\rho_{(h,k,l)}^*
		\]
		which is satisfied since $\de_{\rho_l} \rho^*_{(h,k,j)}=-\de_{\rho_j}\rho_{(h,k,l)}^*$.
		\item $d=0$, $e=2$. Equation \eqref{Afg} becomes
		\[
		-2\de_t (t^n \de_t t^m) \rho_J^*=0-(2t^n \de_t^2 t^m +2 \de_t t^n \de_t t^m) \rho_J^*.
		\]
		\item $d=0$, $e=3$. Equation \eqref{Afg} becomes
		\[
		(2t^n \de_t t^m+ \de_t t^n t^m) \rho_J^*=0+ (2t^n \de_t t^m + \de_t t^n t^m) \rho_J^*.
		\]
		\item $I=(h)$, $J=(k,l)$ with $h\neq k,l$.  Equation \eqref{Afg} becomes
		\[
		t^n \de_t t^m \rho_{(h,k,l)}^*=0+t^n \de_t t^m \de_{\rho_h} \rho_{(k,l)}^*.
		\]
		\item $I=(h)$, $J=(h,k,l)$ with $h\neq k,l,j$. Equation \eqref{Afg} becomes
		\[
		\de_t t^{n+m} \rho_{(k,l)}^*=0+ (t^n \de_t t^m +\de_t t^n t^m) \rho_h \rho_{(h,k,l)}^*.
		\]
		\item $I=(h,k)$, $J=(h,l,j)$ with $k\neq l,j$.  Equation \eqref{Afg} becomes
		\[
		t^{n+m} \rho_{(k,l,j)}^*=0+t^n t^m (-\rho_h)\de_{\rho_k}\rho_{(h,l,j)}^*.
		\]
		\item $d=e=2$ with $I\cap J=\emptyset$. In this case Equation \eqref{Afg} is trivial. 
	\end{itemize}

				This completes the proof in the case where $f$ and $g$ satisfy conditions a) and b).
				Before proceeding with the remaining cases we make some observations. 
				\begin{itemize}
					\item $A(f)\in L_{n+3}$ and $A(g)\in L_{m+3}$;
					\item $[f,g]\in L_{\geq n+d+m+e-2}$;
					\item $[A(f),A(g)]\in L_{\geq n+m+4}$;
					\item $[f,A(g)]\in L_{\geq n+d+m+1}$:
					\item $[A(f),g]\in L_{\geq n+m+e+1}$.
				\end{itemize}
				
				Now we assume that $(f,g)$ satisfies a) but not b), i.e. $e=3$ and $|I\cap J|<|I\cap J^c|$.
				We have that $A(g)\neq 0$ and $(f,A(g))$ satisfies a) and b), and in particular we have
				\begin{equation}\label{Aproperty}
					A([f,A(g)]+[A(f),g])=[f,g]+[A(f),A(g)].
				\end{equation}
				If $n+d+m\geq 2$ we have $[f,A(g)]+[A(f),g]\in L_{\geq 3}$ and the result follows by applying $A$ to \eqref{Aproperty}, recalling that $A^2=I$ on $L_{\geq 3}$. 
				
				If $n+d+m<2$ then necessarily $d=1$ (otherwise $(f,g)$ would satisfy condition b)), $n=m=0$ and $I\cap J=\emptyset$ and so $(f,g)$ is an exceptional pair. Indeed in this case we have  $[f,g]=[A(f),A(g)]=0$ and the result again follows.
				Therefore we can drop assumption b) and deduce that the result holds for all $d,e\leq 3$.
				
				Now assume $d\geq 4$ and $e\leq 3$. Then $(A(f),g)$ is not exceptional and the result holds for the pair $(A(f),g)$, i.e. we have that \eqref{Aproperty} holds. We have $[f, A(g)]\in L_{\geq 3}$; if $n+m+e\geq 2$ we also have $[A(f),g]\in L_{\geq 3}$ and the result follows. We are left with the case $n+m+e<2$, i.e. $(f,g)=(t\rho_I,1),(\rho_I,t),(\rho_I, \rho_k), (\rho_I,1)$. If $(f,g)$ is not exceptional then $[A(f),g]\in L_3$ and the result follows by applying $A$ to \eqref{Aproperty}.
				If $(f,g)$ is exceptional then either $(f,g)=(\rho_I,\rho_k)$ with $k\notin I$ or $(f,g)=(\rho_I,1)$. In both cases $A(g)=0$ and $[f,g]=0$ and the result also follows.
				Finally, we are left with the case $d,e\geq 4$. In this case we have $(A(f),g)$ is not exceptional and so \eqref{Aproperty} holds. Moreover both $[f,A(g)]$ and $[f,A(f)]$ belong to $L_{\geq 3}$ and the result follows by applying $A$ to Equation \eqref{Aproperty} and the proof is complete.
\end{proof}

\begin{proposition}\label{commutator} Let $I$ and $J$ be sequences with coefficients in $\{1,2,3,4,5,6\}$ and let $f=t^n\rho_I$, $g=t^m\rho_J$ be monomials  in $K(1,6)$. Then
\[[\iota(f),\iota(g)]=\begin{cases}
\iota\Big([f,g]+[A(f), A(g)]\Big) & \text{if}\,\,\, [f,g]+[A(f), A(g)]\neq 0,\\
\iota\Big([A(f),g]+[f,A(g)]\Big) & \text{if}\,\,\, [f,g]+[A(f), A(g)]=0.
\end{cases}
\]
Moreover, for all $f,g\in K(1,6)$ (not necessarily monomials) we have
\[
[\iota(f), \iota(tg)]=\iota \Big([f,tg]+[A(f), A(tg)]\Big).
\]
\end{proposition} 
\begin{proof}
The first part follows immediately from Lemma \ref{propertiesofA}. The second part also follows from Lemma \ref{propertiesofA} observing that if $(f,g)$ is an exceptional pair then both $f$ and $g$ are not divisible by $t$. 
\end{proof}
	
\begin{theorem}\label{4.6}
	The image of $\iota$ is a Lie subalgebra of $K(1,6)$ isomorphic to $E(1,6)$.
\end{theorem}
\begin{proof}
By Proposition \ref{commutator} the image of the map $\iota$ is a subalgebra of the Lie superalgebra $K(1,6)$. In order to show that it is isomorphic to $E(1,6)$ it is convenient to use the principal grading, i.e., the grading of type $(2|1,1,1,1,1,1)$, and the variables $\rho_i$ which are homogeneous in this grading. We first notice that the map $\iota$ preserves the principal grading  of $K(1,6)$ and that its image has local part isomorphic to that of $E(1,6)$, i.e., $\iota(K(1,6))_{-1}\oplus\iota(K(1,6))_{0}\oplus\iota(K(1,6))_{1}
\cong E(1,6)_{-1}\oplus E(1,6)_0\oplus E(1,6)_1$. Since the principal grading of $E(1,6)$ is transitive, by \cite[Proposition 5]{K1} it is therefore sufficient to check that the positive part of $\iota(K(1,6))$ is generated by elements of degree one. 
To this aim we notice that for every element
$f=t^n\rho_I\in K(1,6)_k$ with $k>1$, there exist $i\in \{1,2,3,4,5,6\}$ and $g\in K(1,6)_{k-1}$ such that
$f=[g,t\rho_i]$, so that $\iota(f)=\iota([g,t\rho_i])=[\iota(g),\iota(t\rho_i)]$ by Proposition \ref{commutator}, since $A(t\rho_i)=0$. Indeed, we have:
$t^m\rho_I=(-1)^{(|I|+1)}[t^{m-1}\rho_i\rho_I, t\rho_i]$ for some $i\notin I$, and
$t^m\rho_I=-\frac{1}{m+3}[t^{m}\rho_1\dots\rho_5, t\rho_6]$ if
$I=\{1,2,3,4,5,6\}$. An inductive argument concludes the proof.
\end{proof}

\begin{remark}
	\begin{enumerate}Let $I$ be a sequence with distinct entries in $\{2,3,4,\bar 2, \bar 3, \bar 4\}$ and $n\geq 0$;
		\item[i)] if $|I|\geq 3$ then $A(t^n \xi_I)\neq 0$;
		\item[ii)] if $A(t^n \xi_I)\neq 0$ then $A(A(t^n \xi_I))=t^n \xi_I$ and in particular $\iota(t^n \xi_I)=\iota(A(t^n \xi_I))$;
		\item[iii)] the previous observations imply that $E(1,6)$ is spanned by the elements $\iota(t^n \xi_I)$ with $|I|\leq 3$;
		\item[iv)] if $I\in \{( 2,  3,  4),(3, \bar 2,\bar 4,), (2, \bar 3,\bar 4), (4, \bar 2,\bar 3)\}$ then $A(t^n \xi_I)=-t^n \xi_I$ and hence $\iota(t^n \xi_I)=0$.
	\end{enumerate}
\end{remark}
It follows from the previous remark that the set
\[
\{\iota(t^n \xi_I):\,|I|\leq 3 \}
\}\]
is a spanning set for $E(1,6)$. 

We observe that if $|I|\leq 2$ then $\iota(\xi_I)=\xi_I$. The computation of all $\iota(\xi_I)$ with $|I|=3$ is shown in Table \ref{table:iota}. 
\begin{table}[h]
\begin{center}
	\begin{tabular}{|c|c|}
		\hline
		\rule[-1mm]{0mm}{5mm}
		$f$&$\iota(f)$\\
		\hline
		\hline
		\rule[-1mm]{0mm}{5mm}
		$\eta_i \eta_j \eta_k$& $2 \eta_i \eta_j \eta_k$\\
		\hline
		\rule[-1mm]{0mm}{5mm}
		$\xi_j \eta_j \eta_i $& $\xi_j\eta_j \eta_i +\xi_k \eta_k \eta_i $\\
		\hline
		\rule[-1mm]{0mm}{5mm}
		$\xi_{ i} \eta_{j} \eta_{k}$& 0\\
		\hline
		\rule[-1mm]{0mm}{5mm}
		$ \xi_j \xi_{ i}\eta_{i}$& $\xi_{ j} \xi_{ i} \eta_{i}-\xi_{j} \xi_{k} \eta_{k}$\\
		\hline
		\rule[-1mm]{0mm}{5mm}
		$\xi_{i} \xi_{j} \eta_{k}$ & 2 $\xi_{i} \xi_{ j} \eta_{k}$\\
		\hline
		\rule[-1mm]{0mm}{5mm}
		$\xi_{i} \xi_{ j} \xi_{k}$& 0\\	
		\hline
	\end{tabular}
\end{center}
\vspace{2mm}
\caption{Computation of $\iota$ on elements $\xi_I$ with $|I|=3$.  Numbers $i,j,k$ represent any cyclic permutation of $2,3,4$.}\label{table:iota}
\end{table}

We are going to define an explicit linear map $\Psi:E(1,6)\rightarrow E(4,4)$. The values of the map $\Psi$ on all elements $\iota(\xi_I)\neq 0$, where $I$ is any sequence with coefficients in $\{2,3,4,\bar{2},\bar{3},\bar{4}\}$ with $|I|\leq 3$, are shown in Table \ref{defnpsi}. 
\begin{table}[h]
	\begin{center}
		\begin{tabular}{|c|c|c|}
			\hline
			\rule[-1mm]{0mm}{6mm}
			$f$&$\iota(f)$& $\Psi(\iota(f))$\\
			\hline
			\hline
			\rule[-1mm]{0mm}{6mm}
			$1$& $1$ &  $\de_{x_1}$\\
			\hline
			\rule[-1mm]{0mm}{6mm}
			$\eta_i\eta_j$ & $\eta_i\eta_j$& $\de_{x_k}$\\
			\hline
			\rule[-2mm]{0mm}{7mm}
			$\eta_i$& $\eta_i$& $\frac{\sqrt 2}{2}dx_i$\\
			\hline
		\rule[-1mm]{0mm}{6mm}
			$\eta_i \eta_j \eta_k $ & $2 \eta_i \eta_j \eta_k $ & $-\sqrt{2}dx_1 $ \\
			\hline
			\rule[-1mm]{0mm}{6mm}
			$\xi_{i}\eta_j $ & $\xi_{i}\eta_j $ & $x_j \de_{x_i} $ \\
			\hline
			\rule[-1mm]{0mm}{6mm}
			$\xi_{i}\eta_i $ & $\xi_{i}\eta_i $ & $-x_j \de_{x_j}-x_k \de_{x_k} $ \\
			\hline
			\rule[-2mm]{0mm}{7mm}
			$\xi_{i} $ & $\xi_{i} $ & $ \frac{\sqrt 2}{2}(-x_jdx_k+x_kdx_j) $ \\
			\hline
			\rule[-1mm]{0mm}{6mm}
			$\xi_j \eta_j \eta_{i} $ & $\xi_j \eta_j \eta_{i}+\xi_k \eta_k \eta_{i} $ & $\sqrt 2 x_i dx_1 $ \\
			\hline
			\rule[-1mm]{0mm}{6mm}
			$\xi_{i} \xi_{j} $ & $\xi_{i} \xi_{j} $ & $x_k(x_2\de_{x_2}+x_3 \de_{x_3}+x_4 \de_{x_4}) $ \\
			\hline
			\rule[-1mm]{0mm}{6mm}
			$\xi_j  \xi_{i}\eta_{i} $ & $\xi_j  \xi_{i} \eta_{i}-\xi_j \xi_{k} \eta_{k} $ & $\sqrt 2 x_ix_k dx_1 $\\
			\hline
			\rule[-1mm]{0mm}{6mm}
			$\xi_i \xi_j \eta_k $ & $ 2\xi_i \xi_j \eta_k$ & $-\sqrt 2 x_k^2 dx_1 $\\
			\hline
		\end{tabular}
	\end{center}
	\vspace{2mm}
	\caption{Definition of $\Psi$ on elements $\iota(\xi_I)\neq 0$, with $|I|\leq 3$. Numbers $i,j,k$ represent any cyclic permutation of $2,3,4$.}
	\label{defnpsi}
\end{table}
\begin{definition}\label{def:psi}
	We define the map $\Psi:E(1,6)\rightarrow E(4,4)$ as the unique linear map such that for all sequence $I$ with coefficients in $\{2,3,4,\bar{2},\bar{3},\bar{4}\}$ with $|I|\leq 3$ we have $\Psi(\iota(\xi_I))$ is defined as in Table \ref{defnpsi} and, for all $n\geq 1$,
	\[
	\Psi(\iota(t^n \xi_I))=2^nx_1^n \Psi(\iota(\xi_I)).
	\]
\end{definition}
Next target is of course to prove that  $\Psi$ is indeed an embedding of Lie superalgebras and for this we need the following preliminary results.

\begin{lemma}\label{lemmahomo1}
	For all $I,J$ sequences with coefficients in $\{2,3,4,\bar 2, \bar 3, \bar 4\}$ with $|I|,|J|\leq 3$ and all $n\geq 0$ we have
	\[
	[x_1^n \Psi(\iota(\xi_I)),\Psi(\iota(\xi_J))]=(1-n)x_1^n[ \Psi(\iota(\xi_I)),\Psi(\iota(\xi_J))]+nx_1^{n-1}[x_1\Psi(\iota(\xi_I)),\Psi(\iota(\xi_J))].
	\]
\end{lemma}
\begin{proof}
	Recalling the definition of the bracket in $E(4,4)$, since the elements $\Psi(\iota(\xi_I))$ and $\Psi(\iota(\xi_J))$ have coefficients that do not involve the variable $x_1$ (see Table \ref{defnpsi}) we deduce that
	\[[x_1^n \Psi(\iota(\xi_I)),\Psi(\iota(\xi_J))]=x_1^n \alpha+nx_1^{n-1} \beta\]
	for suitable elements $\alpha,\beta\in E(4,4)$ that do not depend on $n$ and whose coefficients do not involve the variable $x_1$. One can obtain $\alpha=[ \Psi(\iota(\xi_I)),\Psi(\iota(\xi_J))]$ by letting $n=0$ and $\beta=[ x_1\Psi(\iota(\xi_I)),\Psi(\iota(\xi_J))]-[ \Psi(\iota(\xi_I)),\Psi(\iota(\xi_J))]$ by letting $n=1$.
\end{proof}
\begin{lemma}\label{lemmahomo2}
	Let $I,J$ be sequences with coefficients in $\{2,3,4,\bar 2, \bar 3, \bar 4\}$ with $|I|,|J|\leq 3$ and $n\geq 1$ such that $[A(t^n\xi_I),A(\xi_J)]=0$. Then
	\[
	[\iota(t^n \xi_I),\iota(\xi_J)]=\iota\big((1-n)t^n[\xi_I,\xi_J]+nt^{n-1}[t\xi_I,\xi_J]\big).
	\]
\end{lemma}
\begin{proof} Since $n\geq 1$, Proposition \ref{commutator} provides
	\[[\iota(t^n\xi_I),\iota(\xi_J)]=\iota([t^n\xi_I,\xi_J])+\iota([A(t^n\xi_I),A(\xi_J)])=\iota([t^n\xi_I,\xi_J]).
	\]

The result now follows from the following identity
\[
[t^n\xi_I,\xi_J]=(1-n)t^n[\xi_I,\xi_J]+nt^{n-1}[t\xi_I,\xi_J],
\]
which is an immediate consequence of the following formula for the bracket in $K(1,6)$ which is valid for all $n\geq 0$
\[
[t^n\xi_I,\xi_J]=-nt^{n-1}(2-|J|)\xi_I\xi_J +(-1)^{|I|}t^n \sum_{i=1}^{3} \big(\de_{\xi_i}\xi_I \de_{\eta_i}\xi_J+ \de_{\eta_i}\xi_I \de_{\xi_i}\xi_J\big).
\]
\end{proof}
\begin{theorem}\label{4.7}
	The map $\Psi:E(1,6)\rightarrow E(4,4)$ given in Definition \ref{def:psi} is an embedding of Lie superalgebras.
\end{theorem}
\begin{proof}
We have to show that
\begin{equation}
	[\Psi(\iota(f)),\Psi(\iota(g))]=\Psi([\iota(f),\iota(g)]) \label{eqn:brack}
\end{equation}
for all elements $\iota(f),\iota(g)\in E(1,6)$. By transitivity of the principal grading of $E(4,4)$ it is enough to show \eqref{eqn:brack} with $\deg g=-1$, i.e., we can assume that
\begin{itemize}
	\item $f=t^n\xi_I$ for some $n\geq 0$ and for some sequence $I$ with coefficients in $\{2,3,4,\bar{2},\bar{3},\bar{4}\}$ such that $|I|\leq 3$;
	\item $g=\eta_J$ for some sequence $J$ with coefficients in $\{2,3,4\}$.
\end{itemize}
Now we observe that many such pairs satisfy the conditions in Lemma \ref{lemmahomo2}, namely, $[A(t^n\xi_I),A(\eta_J)]=0$. Indeed, we first recall that if $|J|\leq 2$ then $A(\eta_J)=0$, so the only case left to consider is $\eta_J=\eta_2\eta_3\eta_4$ which satisfy $A(\eta_J)=\eta_J$.
In this case we can check that $[A(t^n\xi_I),A(\eta_J)]\neq 0$ in the following two cases only: $\xi_I=\xi_j\xi_i \eta_i$ or $\xi_I=\xi_j \xi_i$ for some $i,j\in \{2,3,4\}$, $i\neq j$. We treat both cases explicitly.
In the first case, for the sake of simplicity, let us assume  $\xi_I=\xi_2\xi_3\eta_3$. We have
\begin{align*}
[\Psi(\iota(t^n\xi_2\xi_3\eta_3)),\Psi(\iota(\eta_2\eta_3\eta_4))]&=
2^{n+1}[x_1^n\Psi(\iota(\xi_2\xi_3\eta_3)),\Psi(\eta_2\eta_3\eta_4)]\\
&=2^{n+1}[x_1^n\Psi(\xi_2\xi_3\eta_3-\xi_2\xi_4\eta_4),\Psi(\eta_2\eta_3\eta_4)]\\
&=-2^{n+2}[x_1^nx_3x_4dx_1, dx_1]=0.
\end{align*}
On the other hand,
\begin{align*}
[\iota(t^n\xi_2\xi_3\eta_3),\iota(\eta_2\eta_3\eta_4)]&=
[t^n\xi_2\xi_3\eta_3, \eta_2\eta_3\eta_4]+[t^n\xi_2\xi_4\eta_4, \eta_2\eta_3\eta_4]\\
&=-t^n\xi_2\eta_3\eta_2\eta_4-t^n\xi_2\eta_4\eta_2\eta_3=0,
\end{align*}
hence $\Psi([\iota(t^n\xi_2\xi_3\eta_3),\iota(\eta_2\eta_3\eta_4)])=0$.

In the latter case, we let assume $\xi_I=\xi_3 \xi_4$. We have
\[
[t^n\xi_3\xi_4,\eta_2\eta_3\eta_4]=nt^{n-1}\xi_3\xi_4\eta_2\eta_3\eta_4-t^n(\xi_4\eta_2\eta_4+\xi_3\eta_2 \eta_3)
\]
and
\[
[A(t^n\xi_3\xi_4),\eta_2\eta_3\eta_4]=[-nt^{n-1}\xi_2\xi_3\xi_4\eta_2),\eta_2\eta_3\eta_4]=-nt^{n-1} \xi_3 \xi_4 \eta_2\eta_3\eta_4
\]
and therefore
\[[\iota(t^n\xi_3\xi_4),\iota(\eta_2\eta_3\eta_4)]=-2t^n(\xi_4\eta_2\eta_4+\xi_3\eta_2 \eta_3)=-2\iota(t^n \xi_4\eta_2\eta_4)\]
Finally
\[
\Psi([\iota(t^n\xi_3\xi_4),\iota(\eta_2\eta_3\eta_4)])=-2 \Psi(\iota(t^n\xi_4\eta_2\eta_4))=-2^{n+1}x_1^n \Psi(\iota(\xi_4\eta_2\eta_4))=2^{n+1}\sqrt 2 x_1^nx_2 dx_1.\]
On the other hand
\begin{align*}
	[\Psi(\iota(t^n\xi_3\xi_4)), \Psi(\iota(\eta_2\eta_3\eta_4))] &= 2^n[x_1^n\Psi(\iota(\xi_3\xi_4)), -\sqrt 2 dx_1]\\
	&= 2^n[x_1^n x_2(x_2\de_{x_2}+x_3\de_{x_3}+x_4\de_{x_4}),-\sqrt 2 dx_1 ]\\
	&= 2^{n+1}\sqrt 2 x_1^n x_2 dx_1.
\end{align*} 
Therefore we can assume that $[A(t^n\xi_I),A(\eta_J)]=0$ and next target of the proof is to show that if \eqref{eqn:brack} holds for $n=0,1$ then it holds for all $n\geq 0$. Indeed the lefthand side of \eqref{eqn:brack} can be rewritten by Lemma \ref{lemmahomo1} as
\begin{align*}
[\Psi(\iota(t^n\xi_I)),\Psi(\iota(\xi_J))]&=[2^nx_1^n \Psi(\iota(\xi_I)), \Psi(\iota(\xi_J))]\\
&= 2^n (1-n)x_1^n[ \Psi(\iota(\xi_I)),\Psi(\iota(\xi_J))]+n2^n x_1^{n-1}[x_1\Psi(\iota(\xi_I)),\Psi(\iota(\xi_J))]\\
&=2^n (1-n)x_1^n[ \Psi(\iota(\xi_I)),\Psi(\iota(\xi_J))]+n2^{n-1} x_1^{n-1}[\Psi(\iota(t\xi_I)),\Psi(\iota(\xi_J))] .
\end{align*}
On the other hand, by Lemma \ref{lemmahomo2}, the righthand side can be rewritten as
\begin{align*}
	\Psi([\iota(t^n\xi_I),\iota(\xi_J)])&=\Psi(\iota((1-n)t^n[\xi_I,\xi_J]+nt^{n-1}[t\xi_I,\xi_J]) )\\
	&= 2^n(1-n)x_1^n\Psi(\iota([\xi_I,\xi_J]))+n2^{n-1}x_1^{n-1}\Psi(\iota([t\xi_I,\xi_J])).
\end{align*}
It remains to verify that \eqref{eqn:brack} is satisfied for $f=\xi_I, t\xi_I$ and $g=\eta_J$ and this has been done with the help of a computer.
\end{proof}

This result allows us to locate a subalgebra of $E(4,4)$ isomorphic to $E(1,6)$.  Let
\[
V_0=\langle \{\de_1,\de_i,x_i\de_j, x_i(x_2\de_2+x_3\de_3+x_4\de_4):\, i,j=2,3,4\}\rangle \subseteq E(4,4)_{\bar 0}
\]
and 
\[
V_1=\langle \{dx_1,dx_i, x_idx_j-x_jdx_i,x_idx_1,x_i x_jdx_1:\, i,j=2,3,4\}\rangle \subseteq E(4,4)_{\bar 1}
\]
be subspaces of $E(4,4)$.
\begin{corollary}
	The subspace
	\[ \C[[x_1]](V_0+V_1)\]
	of $E(4,4)$ is indeed a subalgebra of $E(4,4)$ isomorphic to $E(1,6)$.
\end{corollary}

\end{document}